\documentclass[draft]{jalc}
\usepackage{graphicx}
\usepackage{tikz}
\usepackage{caption}
\usepackage{amsmath, amssymb}

\newcommand{\Cartesian}[2]{#1\,\square\,#2}
\newcommand{\fle}{x}
\newcommand{\sle}{y}
\newcommand{\tle}{z}
\newcommand{\fol}{x'}
\newcommand{\fwo}{w}

\newtheorem{defn}{Definition}
\newtheorem{thm}[defn]{Theorem}
\newtheorem{lem}[defn]{Lemma}

\usepackage{tikz}
\begin{document}

\title{The $k$-cube is $k$-representable}
\runningtitle{The $k$-cube is $k$-representable}
\runningauthors{\textsc{B.~Broere}, \textsc{H.~Zantema}}

  \author[RU]{Bas Broere}
\address[RU]{Radboud University Nijmegen, P.O.\ Box 9010, \\ 6500 GL Nijmegen, The Netherlands
  \\
  \email{broerebas@gmail.com}}
\author[TUE,RU]{Hans Zantema}
\address[TUE]{Department of Computer Science, TU Eindhoven, P.O.\ Box 513,\\
		5600 MB Eindhoven, The Netherlands
		\\
		\email{h.zantema@tue.nl}}


\maketitle

\begin{abstract}
A graph is called $k$-representable if there exists a word $\fwo$ over the nodes of the graph, each node occurring exactly $k$ times, such that there is an edge between two nodes $\fle,\sle$ if and only after removing all letters distinct from $\fle,\sle$, from $\fwo$, a word remains in which $\fle,\sle$ alternate.
We prove that if $G$ is $k$-representable for $k>1$, then the Cartesian product of $G$ and the complete graph on $n$ nodes is $(k+n-1)$-representable.
As a direct consequence, the $k$-cube is $k$-representable for every $k \geq 1$.

Our main technique consists of exploring occurrence based functions that replace every $i$th occurrence of a symbol $\fle$ in a word $\fwo$ by a string $h(\fle,i)$.
The representing word we construct to achieve our main theorem is purely composed from concatenation and occurrence based functions.
\end{abstract}

\section{Introduction}
For a word $\fwo$ over an alphabet $A$, the two letters $\fle$ and $\sle$ are said to \emph{alternate} in $\fwo$ if between every two $\fle$'s in $\fwo$ a $\sle$ occurs and between every two $\sle$'s in $\fwo$ an $\fle$ occurs. Stated otherwise: deleting all letters but $\fle$ and $\sle$ from $\fwo$ results in a word $\fle\sle\fle\sle\hdots$ or $\sle\fle\sle\fle\hdots$ of even or odd length.

A graph $G=(V,E)$ is \emph{word-representable} if there is a word $\fwo$ over the alphabet $V$, such that $(\fle,\sle)\in E$ if and only if $\fle$ and $\sle$ alternate in $\fwo$.
The word $\fwo$ is said to \emph{represent}, or be a \emph{representant} of, $G$.
A word only represents one graph, while a graph can have multiple words representing it.

A lot of work has been done on investigating which graphs are word-representable; this is a main topic of the book \cite{book:kitloz}. More recent work includes \cite{art:rep-cro,art:132-rep,art:pat-avo,art:nea-tri}.
A first basic observation is that one may restrict to \emph{uniform} words: a word $\fwo$ over an alphabet $A$ is called \emph{uniform} if there exists a number $k$ such that every letter in $A$ occurs exactly $k$ times in $\fwo$.
For such $k$, the word $\fwo$ is called \emph{$k$-uniform}. So the basic observation states
\begin{thm}(\cite{art:kitpya})
A graph $G$ is representable if and only if it is representable by a $k$-uniform word for some $k\ge 1$.
\end{thm}
In this case $G$ is called $k$-representable. The minimum $k$ such that there exists a $k$-uniform word representing a graph $G$ is called the \emph{representation number} of the graph $G$; for a non-word-representable graph its representation number is defined to be $\infty$.

Now it is a natural question which representation numbers occur, and what are the representation numbers of particular graphs. Lots of results in this direction are given in the book \cite{book:kitloz}. A graph of particular interest is the $k$-cube $Q_k$, so an obvious question is to establish the representation number of $Q_k$. The nodes of the $k$-cube $Q_k$
are the $2^k$ Boolean vectors of length $k$, and two such nodes are connected by an edge if and only if they differ in exactly one position. Equivalently, we can define inductively
$Q_1 =K_2$ consisting of two nodes connected by an edge, and for $k>1$ the $k$-cube $Q_k$ is defined to be the Cartesian product of $K_2$ and $(k-1)$-cube. The Cartesian product we present in more detail in Section \ref{secCp}. In this paper we answer this question in one direction: we show that the representation number of the $k$-cube is at most $k$ by constructing a $k$-uniform word for which show that it represents the $k$-cube. In fact it is an instance of a more general construction: for every graph $G$ represented by a $k$-uniform word, we construct a $(k+1)$-uniform word representing the Cartesian product of $G$ and $K_2$.

After having done this we generalize this further: for $K_n$ being the complete graph on $n>1$ nodes, for every graph $G$ represented by a $k$-uniform word, we construct a $(k+n-1)$-uniform word representing the Cartesian product of $G$ and $K_n$.

Our constructions are based {\em occurrence based functions}: functions $h$ on $k$-uniform words in which every $i$th occurrence of a letter $\fle$ is mapped to a fixed string $h(\fle,i)$. All our constructions are just concatenations of occurrence based functions.

The paper is organized as follows. First, in Section \ref{secprel} we give some basic notations and preliminaries on occurrence-based functions.
In Section \ref{secCp} we define Cartesian products and present the main result on taking the product with $K_2$.
In Section \ref{seccubespr} we discuss its consequences for cubes and prisms. Next,
in Section \ref{secext} we extend our main result to the product with $K_n$ for arbitrary $n$. We conclude in Section \ref{secconcl}.

\section{Preliminaries}
\label{secprel}
In this section we collect some preliminaries, in particular on our notion of occurrence-based functions, and some convenient notations.

\begin{defn}(\cite{book:kitloz})
If $\fwo$ is a word over an alphabet $A$, and $B \subseteq A$, then the word $\fwo_B$ is defined to be obtained by removing all letters in $A\setminus B$ from $\fwo$.
\end{defn}

So two letters $\fle$ and $\sle$ alternate in $\fwo$ if and only if $\fwo_{\{\fle,\sle\}}$ is either $\fle\sle\fle\sle\hdots$ or $\sle\fle\sle\fle\hdots$,
and two letters $\fle$ and $\sle$ alternate in a $k$-uniform word $\fwo$ if and only if $\fwo_{\{\fle,\sle\}}$ is either $(\fle\sle)^k$ or $(\sle\fle)^k$.

We now introduce the notion of occurrence-based functions.
These functions appear to be very useful in constructing representants for word-representable graphs.

\begin{defn}
Let $V$ and $V'$ be (possibly different) alphabets, and let $N_k=\{1,\hdots,k\}$.

The \emph{labelling function} of a word over $V$ is defined as $H:(V^n)^* \rightarrow (V\times N_k)^*$, where the $i$th occurrence of each letter $\fle$ is mapped to the pair $(\fle,i)$, and $k$ satisfies the property that every symbol occurs at most $k$ times in $\fwo$.
The word $H(\fwo)$ is called the \emph{labelled version of $\fwo$}.

An \emph{occurrence-based function} is defined as applying a string homomorphism $h : V\times N_k \rightarrow (V')^*$ to an already labelled version of a word.
As a shorthand we will write $h(\fwo)$ instead of $h(H(\fwo))$.
\end{defn}

\begin{defn}
For a $k$-uniform word $w$ and a non-empty set $A \subseteq N_k=\{1,\hdots,k\}$ the occurrence based function $p_A$ is defined by $p_A(\fle, i) = \fle$ for all $i \in A$,
and $p_A(\fle,i)= \epsilon$ for all $i \not\in A$, for every symbol $\fle$.
\end{defn}

For instance, $p_{\{i\}}(\fwo)$ is obtained by removing all but the $i$th occurrence of each symbol in $\fwo$; this is called the
\emph{$i$th permutation} of a word $\fwo$.
Clearly, if $w$ is $k$-uniform, then $p_A(\fwo)$ is $\#A$-uniform, and if $w_{\{\fle,\sle\}} = (\fle\sle)^k$ then $p_A(\fwo)_{\{\fle,\sle\}} = (\fle\sle)^{\# A}$

\begin{lem}\label{lem}
Let $\fwo$ be a $k$-uniform word representing a graph $G$. For some $m > 1$ let $A_1,\ldots,A_m$ be non-empty subsets of $N_k=\{1,\hdots,k\}$
such that for all $j=1,\ldots,k-1$ there is $i \in \{1,\ldots,m\}$ such that $\{j,j+1\} \subseteq A_i$. Then the $(\sum_{i=1}^m \# A_i)$-uniform
word $\fwo' = p_{A_1}(\fwo) p_{A_2}(\fwo) \cdots p_{A_m}(\fwo)$ also represents the graph $G$.
\end{lem}

\begin{proof}
We have to prove that any two symbols $\fle,\sle$ alternate in $\fwo$ if and only if they alternate in $\fwo'$.

First assume they alternate in $\fwo$, then $\fwo_{\{\fle,\sle\}}$ is either $(\fle \sle)^k$ or $(\sle \fle)^k$. Assume it is $(\fle \sle)^k$,
the other case is similar by swapping $\fle$ and $\sle$. Then $p_{A_i}(\fwo)_{\{\fle,\sle\}} = (\fle\sle)^{\# A_i}$ for all $i = 1,\ldots,m$, so
$\fwo'_{\{\fle,\sle\}} = (\fle\sle)^{\sum_{i=1}^m \# A_i}$, by which $\fle,\sle$ alternate in $\fwo'$.

Conversely, assume that $\fle,\sle$ alternate in $\fwo'$. Then either $p_{A_i}(\fwo)_{\{\fle,\sle\}} = (\fle\sle)^{\# A_i}$ for all $i = 1,\ldots,m$, or
$p_{A_i}(\fwo)_{\{\fle,\sle\}} = (\sle\fle)^{\# A_i}$ for all $i = 1,\ldots,m$; let's assume the first, the other case is similar by swapping $\fle$ and $\sle$.
Let $\{j,j+1\} \subseteq A_i$, then from $p_{A_i}(\fwo)_{\{\fle,\sle\}} = (\fle\sle)^{\# A_i}$ we conclude that
\begin{itemize}
\item the $j$th $\fle$ in $\fwo$ is left from the $j$th $\sle$ in $\fwo$,
\item the $j$th $\sle$ in $\fwo$ is left from the $(j+1)$th $\fle$ in $\fwo$, and
\item the $(j+1)$th $\fle$ in $\fwo$ is left from the $(j+1)$th $\sle$ in $\fwo$.
\end{itemize}
As it is assumed for all $j=1,\ldots,k-1$ there is such an $A_i$, we obtain this
 property for all $j=1,\ldots,k-1$, from which we conclude that $\fle,\sle$ alternate in $\fwo$.
\end{proof}

As a direct consequence of Lemma \ref{lem} we obtain that if $G$ is a graph represented by a $k$-uniform word $\fwo$, then for every $1\le i\le k$
the $(k+1)$-uniform word $p_{\{i\}}(\fwo)\fwo$ also represents $G$.

Now we come to the main topic of this paper: word representations of Cartesian products of graphs.

\section{Cartesian products}
\label{secCp}

\begin{defn}(\cite{book:kitloz})
The \emph{Cartesian product} of two graphs $G = (V_G,E_G)$ and $H = (V_H,E_H)$ is defined as $\Cartesian{G}{H} = (V_{GH},E_{GH})$, where $V_{GH}=V_G\times V_H$ and\\
$E_{GH} = \{\,\left(\,(\fle,\tle)\,,\,(\sle,\fol)\,\right)\,|\,\fle=\sle\text{ and } (\tle,\fol)\in E_H\text{ or }\tle=\fol\text{ and }(\fle,\sle)\in E_G \}$.
\end{defn}

It is known from \cite{book:kitloz} that if both $G$ and $H$ are word-representable, then $\Cartesian{G}{H}$ is word-representable.
However, it is not yet clear how to find a representant of this Cartesian product directly from representations for both $G$ and $H$.

In this section we will give a construction of a word representing a special case of this situation, namely $H = K_2$, the complete graph on 2 nodes.

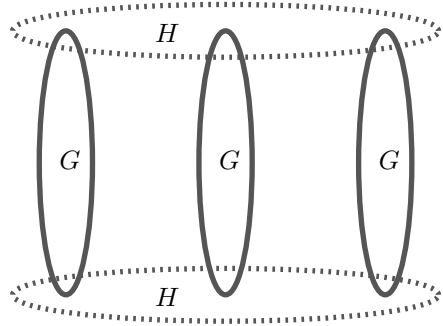
\begin{figure}[htbp!]
    \centering
    \begin{tikzpicture}[x=1cm,y=1cm,scale=0.7]
        \draw [rotate around={90:(-12,4.5)},line width=2pt,color={rgb:red,0.4;green,0.4;blue,0.4}] (-12,4.5) ellipse (2.5cm and 0.5cm);
        \draw [rotate around={90:(-9,4.5)},line width=2pt,color={rgb:red,0.4;green,0.4;blue,0.4}] (-9,4.5) ellipse (2.5cm and 0.5cm);
        \draw [rotate around={90:(-6,4.5)},line width=2pt,color={rgb:red,0.4;green,0.4;blue,0.4}] (-6,4.5) ellipse (2.5cm and 0.5cm);
        \draw [rotate around={0:(-9,2)},line width=2pt,dash pattern=on 1pt off 2pt,color={rgb:red,0.4;green,0.4;blue,0.4}] (-9,2) ellipse (4.0cm and 0.5cm);
        \draw [rotate around={0:(-9,7)},line width=2pt,dash pattern=on 1pt off 2pt,color={rgb:red,0.4;green,0.4;blue,0.4}] (-9,7) ellipse (4.0cm and 0.5cm);
        \draw (-12.3,4.9) node[anchor=north west] {$G$};
        \draw (-9.3,4.9) node[anchor=north west] {$G$};
        \draw (-6.3,4.9) node[anchor=north west] {$G$};
        \draw (-10.5,7.3) node[anchor=north west] {$H$};
        \draw (-10.5,2.3) node[anchor=north west] {$H$};
    \end{tikzpicture}
    \caption{Cartesian product of two graphs, $G$ and $H$.}
    \label{fig:my_label}
\end{figure}

When looking at the Cartesian product of a graph $G$ with the complete graph on $n>1$ nodes, $K_n$, the resulting graph consists of $n$ copies of $G$, in which moreover any two nodes corresponding to the same node in $G$ are connected by an edge.
It is also easily verified that the complete graph on $n$ nodes is represented by the 1-uniform word $\fwo=12\hdots n$.
In fact, the complete graphs are the only graphs of representation number 1, see \cite{art:kitrep3}.\\

The complete graph $K_2$ just consists of tow nodes connected by a single edge; the nodes of $\Cartesian{G}{K_2}$ are denoted by $\fle_1,\fle_2$ for $\fle$ running over the nodes of $G$; two nodes $\fle_i$, $\sle_j$ are connected by an edge in $\Cartesian{G}{K_2}$ if and only if
\begin{itemize}
\item $i=j$ and $(\fle,\sle)$ is an edge in $G$, or
\item $i\neq j$ and $\fle=\sle$.
\end{itemize}
Write $V_1$ for the set of nodes $\fle_1$ and $V_2$ for the set of nodes $\fle_2$, so $V_1 \cup V_2$ is the set of nodes of $\Cartesian{G}{K_2}$.

\begin{thm}\label{thm:main}
Let $G$ be a $k$-representable graph for $k>1$ and let $\fwo$ be a $k$-representant of $G$.
Then the graph $\Cartesian{G}{K_2}$ is $(k+1)$-representable with representant $\fwo' = f(\fwo) g(\fwo)$ for the occurrence based functions $f,g$ defined by
\[ f(\fle,i) = \left\{ \begin{array}{ll}
\fle_1 & \mbox{if $i = 1$} \\
\fle_2 \fle_1 & \mbox{if $1 < i \leq k$} \end{array} \right. \hspace{1cm}
g(\fle,i) = \left\{ \begin{array}{ll}
\fle_2 & \mbox{if $i = 1$ }\\
\fle_1 \fle_2 & \mbox{if $i = 2$} \\
\epsilon & \mbox{if $2 < i \leq k$} \end{array} \right. \]
\end{thm}

\begin{proof}
For every $\fle$ the word $f(\fwo)$ contains $k$ copies of $\fle_1$ and $k-1$ copies of $\fle_2$, and  the word $g(\fwo)$ contains 1 copy of $\fle_1$ and 2 copies of $\fle_2$,
so indeed $\fwo'$ is $(k+1)$-uniform.

We have to prove that $\fle_i, \sle_j$ alternate in $\fwo'$ for $\fle_i \neq \sle_j$ if and only if $(\fle_i, \sle_j)$ is an edge in $\Cartesian{G}{K_2}$, for $i,j = 1,2$, that is
\begin{itemize}
\item if $x \neq y$ and $i=j$ then $\fle_i, \sle_j$ alternate in $\fwo'$ if and only if $\fle,\sle$ alternate in $\fwo$,
\item if $x = y$ and $i\neq j$ then $\fle_i, \sle_j$ do alternate in $\fwo'$, and
\item if $x \neq y$ and $i\neq j$ then $\fle_i, \sle_j$ do not alternate in $\fwo'$.
\end{itemize}
We do this by considering all these cases.

Let $\fle \neq \sle$ and $i=j=1$. Observe that $f(\fwo)_{V_1} = \fwo_1$ and $g(\fwo)_{V_1} = p_2(\fwo_1)$, in which $\fwo_1$ is a copy of $\fwo$ in which every symbol is indexed by 1.
Now $\fle_1,\sle_1$ alternate in $\fwo'$ if and only if they alternate in $\fwo'_{V_1} = f(\fwo)_{V_1} g(\fwo)_{V_1} =  \fwo_1 p_2(\fwo_1)$, and by Lemma \ref{lem} for $A_1 = N_k$ and $A_2 = \{2\}$, with $N_k=\{1,\hdots,k\}$, this holds if and only if $\fle,\sle$ alternate in $\fwo$, which we had to prove.

Let $\fle \neq \sle$ and $i=j=2$. Observe that $f(\fwo)_{V_2} = p_{N_k \setminus \{1\}}(\fwo_2)$ and $g(\fwo)_{V_2} = p_{\{1,2\}}(\fwo_2)$, in which $\fwo_2$ is a copy of $\fwo$ in which every symbol is indexed by 2.
Now $\fle_2,\sle_2$ alternate in $\fwo'$ if and only if they alternate in $\fwo'_{V_2} = f(\fwo)_{V_2} g(\fwo)_{V_2} = p_{N_k \setminus \{1\}}(\fwo_2) p_{\{1,2\}}(\fwo_2)$, and by Lemma \ref{lem} for $A_1 = N_k \setminus \{1\}$ and $A_2 = \{1,2\}$ this holds if and only if $\fle,\sle$ alternate in $\fwo$, which we had to prove.

Let $\fle = \sle$ and $i \neq j$, say $i=1$, $j=2$. Then $\fle_1,\fle_2$ alternate in $\fwo'$ since $\fwo'_{\{\fle_1,\fle_2\}} = \fle_1 (\fle_2 \fle_1)^{k-1} \fle_2 \fle_1 \fle_2 = (\fle_1 \fle_2)^{k+1}$, which we had to prove.

Now, let $\fle \neq \sle$ and $i \neq j$, say $i=1$, $j=2$.

If $\fwo_{\{\fle,\sle\}} = (\fle\sle)^k$ then $f(\fwo)_{\{\fle_1,\sle_2\}}$ starts by $\fle_1 \fle_1$, so $\fle_1,\sle_2$ do not alternate in $\fwo' = f(\fwo) g(\fwo)$.

If $\fwo_{\{\fle,\sle\}} = (\sle\fle)^k$ then $g(\fwo)_{\{\fle_1,\sle_2\}} = \sle_2 \sle_2 \fle_1$, so $\fle_1,\sle_2$ do not alternate in $\fwo' = f(\fwo) g(\fwo)$.

In the remaining case $\fle,\sle$ do not alternate in $\fwo$, so $\fwo_{\{\fle,\sle\}}$ contains either $\fle\fle$ or $\sle\sle$. If it is $\fle\fle$, or it is $\sle\sle$ and $\fwo_{\{\fle,\sle\}}$ does not start in $\sle\sle$, then $f(\fwo)_{\{\fle_1,\sle_2\}}$ contains $\fle_1 \fle_1$. Otherwise $\fwo_{\{\fle,\sle\}}$  starts in $\sle\sle$, but then $g(\fwo)_{\{\fle_1,\sle_2\}} = \sle_2 \sle_2 \fle_1$. In all cases we conclude that
$\fle_1,\sle_2$ do not alternate in $\fwo' = f(\fwo) g(\fwo)$, concluding the proof.
\end{proof}


\section{Cubes and Prisms}
\label{seccubespr}
Theorem \ref{thm:main} has a couple of implications.
In particular, it implies that the $k$-cube $Q_k$ is $k$-representable, as is stated in the following theorem.

\begin{thm}
For every $k\ge 1$, the $k$-cube $Q_k$ is $k$-representable.
\end{thm}

\begin{proof}
The proof is by induction on $k$.
For $k=1$ we observe that $Q_1 = K_2$, being 1-representable by the word $\fwo=12$.

For $k=2$ we observe that $Q_2$ is the 4-cycle being 2-representable by the word $\fwo=31421324$.

For the induction step for $k>2$, we use Theorem \ref{thm:main} giving a $k$-uniform representant for $Q_k$ from a $(k-1)$-uniform representant of $Q_{k-1}$.
\end{proof}

It was already known from \cite{art:kitpya} that every prism is 3-representable and that the 3-prism is not 2-representable.
Theorem \ref{thm:main} also implies that every prism is 3-representable, as a prism is the Cartesian product of a cycle-graph and $K_2$ and cycle-graphs are 2-representable (\cite{book:kitloz}).
From the fact that the 3-prism is not 2-representable and Theorem \ref{thm:main} we can prove the following.

\begin{thm}
The Cartesian product $\Cartesian{K_n}{K_2}$ has representation number $n$ for $n=1,2,3$, and representation number 3 for all $n>3$.
\end{thm}

\begin{proof}
$\Cartesian{K_1}{K_2}$ is equal to $K_2$ having representation number 1.

$\Cartesian{K_2}{K_2}$ is the 4-cycle, which is known to have representation number 2 (\cite{book:kitloz}).

$\Cartesian{K_3}{K_2}$ is equal to the 3-prism, which is not 2-representable, but it is 3-representable (\cite{book:kitloz}).
If $n>3$ then $\Cartesian{K_n}{K_2}$ contains the 3-prism as induced subgraph and thus cannot be 2-representable.
Theorem \ref{thm:main} gives a 3-representation because $K_n$ has a 2-representation $12\cdots n12\cdots n$, so $\Cartesian{K_n}{K_2}$ has representation number 3.
\end{proof}

In particular, this theorem shows that the requirement $k>1$ in Theorem \ref{thm:main} is essential: for $k=1$ the claim of Theorem \ref{thm:main} does not hold since it would yield a non-existent 2-representation of $\Cartesian{K_n}{K_2}$ for $n > 2$.

\section{Extension to $K_n$}
\label{secext}
The ideas used in Theorem \ref{thm:main} can be applied to prove the following generalization.

The nodes of $\Cartesian{G}{K_n}$ are denoted by $\fle_1,\fle_2, \hdots \fle_n$ for $\fle$ running over the nodes of $G$; two nodes $\fle_i$, $\sle_j$ are connected by an edge in $\Cartesian{G}{K_2}$ if and only if
\begin{itemize}
\item $i=j$ and $(\fle,\sle)$ is an edge in $G$, or
\item $i\neq j$ and $\fle=\sle$.
\end{itemize}
Write $V_i$ for the set of nodes $\fle_i$, so $V_1 \cup V_2 \cup \cdots \cup V_n$ is the set of nodes of $\Cartesian{G}{K_n}$.

\begin{thm}\label{thm:extended}
Let $G$ be a $k$-representable graph for $k>1$ and let $\fwo$ be a $k$-representant of $G$.
Then the graph $\Cartesian{G}{K_n}$ is $(k+n-1)$-representable with representant $\fwo' = f_n(\fwo) f_{n-1}(\fwo) \cdots f_1(\fwo)$ for the occurrence based functions $f_i$ defined by
\[
f_1(\fle,i) = \left\{
\begin{array}{ll}
  \fle_1 & \mbox{if $i = 1$} \\
  \fle_n \fle_{n-1} \hdots \fle_1 & \mbox{if $1 < i \leq k$}
\end{array} \right.
\]
and
\[f_j(\fle,i) = \left\{
\begin{array}{ll}
  \fle_j & \mbox{if $i = 1$ }\\
  \fle_{j-1} \hdots \fle_1 \fle_n \hdots \fle_j & \mbox{if $i = 2$} \\
  \epsilon & \mbox{if $2 < i \leq k$}
\end{array} \right.
\]
for $j=2,\ldots,n$.
\end{thm}

\begin{proof}
For every $\fle$ the word $f_1(\fwo)$ contains $k$ copies of $\fle_1$ and $k-1$ copies of $\fle_i$ for $i>2$, and the words $f_j(\fwo)$ contain 2 copies of $\fle_j$ and 1 copy of $\fle_i$ for $i \neq j$.
So for every $i$, $\fle_i$ occurs either $k+(n-1)$ times if $i=1$, or $(k-1)+2+(n-2)=k+(n-1)$ times if $i \neq 1$.
So $\fwo'$ is $(k+(n-1))$-uniform.

We have to prove that $\fle_i, \sle_j$ alternate in $\fwo'$ for $\fle_i \neq \sle_j$ if and only if $(\fle_i, \sle_j)$ is an edge in $\Cartesian{G}{K_n}$, for $i,j = 1,2,\hdots,n$, more precisely:
\begin{itemize}
\item if $x \neq y$ and $i=j$ then $\fle_i, \sle_j$ alternate in $\fwo'$ if and only if $\fle,\sle$ alternate in $\fwo$,
\item if $x = y$ and $i\neq j$ then $\fle_i, \sle_j$ do alternate in $\fwo'$, and
\item if $x \neq y$ and $i\neq j$ then $\fle_i, \sle_j$ do not alternate in $\fwo'$.
\end{itemize}
We do this by considering all cases.

Let $\fle \neq \sle$ and $i=j=1$.
Observe that $f_1(\fwo)_{V_1} = \fwo_1$ and $f_i(\fwo)_{V_1} = p_2(\fwo_1)$ for all $i>1$, in which $\fwo_1$ is a copy of $\fwo$ in which every symbol is indexed by 1.
Now $\fle_1,\sle_1$ alternate in $\fwo'$ if and only if they alternate in $\fwo'_{V_1} = f_n(\fwo)_{V_1} f_{n-1}(\fwo)_{V_1}\hdots f_1(\fwo)_{V_1} =  (p_2(\fwo_1))^{n-1} \fwo_1$, and by Lemma \ref{lem} for $A_1 = A_2 = \hdots = A_{n-1} = \{2\}$ and $A_n = N_k$, with $N_k=\{1,\hdots,k\}$, this holds if and only if $\fle,\sle$ alternate in $\fwo$, which we had to prove.

Let $\fle \neq \sle$ and $i=j \geq 2$.
Observe that $f_1(\fwo)_{V_i} = p_{N_k \setminus \{1\}}(\fwo_i)$ and $f_l(\fwo)_{V_i} = p_{\{1,2\}}(\fwo_i)$ for all $l>1$, in which $\fwo_i$ is a copy of $\fwo$ in which every symbol is indexed by $i$.
Now $\fle_i,\sle_i$ alternate in $\fwo'$ if and only if they alternate in $\fwo'_{V_i} = f_n(\fwo)_{V_i} f_{n-1}(\fwo)_{V_2}\hdots f_1(\fwo)_{V_i} = (p_{\{1,2\}}(\fwo_i))^{n-1} p_{N_k \setminus \{1\}}(\fwo_i)$, and by Lemma \ref{lem} for $A_1 = A_2 = \hdots = A_{n-1} = \{1,2\}$ and $A_n = N_k \setminus \{1\}$ this holds if and only if $\fle,\sle$ alternate in $\fwo$, which we had to prove.

Let $\fle = \sle$, $i = 1$ and $i < j$.
Then $\fle_1,\fle_j$ alternate in $\fwo'$ since $\fwo'_{\{\fle_1,\fle_j\}} = (\fle_j \fle_1)^{n-j} (\fle_j\fle_1\fle_j) (\fle_1\fle_j)^{j-2} (\fle_1 (\fle_j \fle_1)^{k-1})= (\fle_j \fle_1)^{k+(n-1)}$, which we had to prove.

Let $\fle = \sle$, $1 \neq i < j$.
Then $\fle_i,\fle_j$ alternate in $\fwo'$ since $\fwo'_{\{\fle_i,\fle_j\}} = (\fle_j \fle_i)^{n-j} (\fle_j\fle_i\fle_j) (\fle_i\fle_j)^{j-i-1} (\fle_i\fle_j\fle_i)(\fle_j\fle_i)^{i-2} (\fle_j \fle_i)^{k-1}= (\fle_j \fle_i)^{k+(n-1)}$, which we had to prove.

It remains to consider $\fle \neq \sle$ and $i \neq j$; without loss of generality assume $i < j$. We continue by case analysis on the shape of $\fwo_{\{\fle,\sle\}}$.

If $\fwo_{\{\fle,\sle\}} = (\fle\sle)^k$ and $i=1$, then $f_1(\fwo)_{\{\fle_1,\sle_j\}}$ starts by $\fle_1 \fle_1$, so $\fle_1,\sle_j$ do not alternate in $\fwo'$.

If $\fwo_{\{\fle,\sle\}} = (\fle\sle)^k$ and $i>1$, then $f_i(\fwo)_{\{\fle_i,\sle_j\}} = \fle_i \fle_i \sle_j$, so $\fle_i,\sle_j$ do not alternate in $\fwo'$.

If $\fwo_{\{\fle,\sle\}} = (\sle\fle)^k$ then $f_j(\fwo)_{\{\fle_i,\sle_j\}} = \sle_j \sle_j \fle_i$, so $\fle_i,\sle_j$ do not alternate in $\fwo'$.

In the remaining case $\fle,\sle$ do not alternate in $\fwo$, so $\fwo_{\{\fle,\sle\}}$ contains either $\fle\fle$ or $\sle\sle$.
If it is $\fle\fle$, or it is $\sle\sle$ and $\fwo_{\{\fle,\sle\}}$ does not start in $\sle\sle$, then $f_1(\fwo)_{\{\fle_1,\sle_j\}}$ contains $\fle_1 \fle_1$ for all $j>1$ and $f_i(\fwo)_{\{\fle_i,\sle_j\}} = \fle_i \fle_i \sle_j$ for all $1 \neq i < j$.
Otherwise $\fwo_{\{\fle,\sle\}}$  starts in $\sle\sle$, but then $f_j(\fwo)_{\{\fle_i,\sle_j\}} = \sle_j \sle_j \fle_i$ for all $1 \neq i < j$.
In all cases we conclude that $\fle_i,\sle_j$ do not alternate in $\fwo'$, concluding the proof.
\end{proof}

\section{Conclusions}
We introduced occurrence based functions as a building block when given a word representing a graph, to construct a word representing a modified graph based on the given word.
Exploiting the key lemma (Lemma \ref{lem}) we succeeded in doing so for taking the Cartesian product with $K_n$. Doing this only for the Cartesian product with $K_2$, as a consequence
we have a construction for a $k$-uniform word representing the $k$-cube $Q_k$. Hence, the representation number of $Q_k$ is at most $k$. It remains open whether the representation number of $Q_k$ is exactly $k$, for that one should prove that for $m<k$ no $m$-uniform word exists representing $Q_k$.

When considering our approach combining occurrence based functions and concatenation, we see that several (but not all) constructions in \cite{book:kitloz} giving word representations
of particular graphs could be presented in the same format. It would make sense to investigate which building blocks are used in which constructions.

Our Theorem \ref{thm:extended} is expected to have further generalizations, as will be elaborated in the master thesis \cite{mth:Broere} of the first author.
\label{secconcl}


\biblio{refrepgr}

\end{document}